\documentclass[draft]{amsart}
\usepackage{dsfont}
\usepackage{amsfonts}
\usepackage{mathrsfs}
\usepackage{amssymb}
\usepackage{amsmath}
\usepackage{amsfonts}
\usepackage{amsfonts,enumerate}
\usepackage{pifont}
\usepackage{enumerate}
\usepackage{graphics}
\usepackage{verbatim}
\usepackage{amsthm}
\usepackage{latexsym,bm}
\usepackage{color}
\numberwithin{equation}{section}
\newtheorem{theorem}{Theorem}[section]

\newtheorem{lemma}[theorem]{Lemma}
\newtheorem{proposition}[theorem]{Proposition}

\theoremstyle{definition}
\newtheorem{remark}[theorem]{Remark}

\newcommand{\M}{{\mathcal M}}
\newcommand{\R}{\mathbb R}
\newcommand{\A}{\mathcal A}
\newcommand{\B}{\mathcal B}

\begin{document}

\title[]{The behavior of the bounds of matrix-valued maximal inequality in $\mathbb{R}^n$ for large $n$}

\author{Guixiang Hong}
\address{Instituto de Ciencias Matem\'aticas,
CSIC-UAM-UC3M-UCM, Consejo Superior de Investigaciones
Cient\'ificas, C/Nicol\'as Cabrera 13-15. 28049, Madrid. Spain.\\
\emph{E-mail address: guixiang.hong@icmat.es}}

\thanks{\small {{\it MR(2000) Subject Classification}.} Primary
46L51; Secondary 42B25.}
\thanks{\ This work has been supported by MINECO: ICMAT Severo Ochoa project SEV-2011-0087.}

\maketitle

\begin{abstract}
In this paper, we study the behavior of the bounds of matrix-valued
maximal inequality in $\mathbb{R}^n$ for large $n$. The main result of this
paper is that the $L_p$-bounds ($p>1$) can be taken to be
independent of $n$, which is a generalization of Stein and
Str\"omberg's resut in the scalar-valued case. We also show that the
weak type $(1,1)$ bound has similar behavior as Stein and
St\"omberg's.
\end{abstract}

\section{Introduction}
Let $(X,d,\mu)$ be a metric measure space and $\B(\ell_2)$ the
matrix algebra of bounded operators on $\ell_2$. For a locally
integrable $\B(\ell_2)$-valued function $f$, we define
$$f_r(x)=\frac{1}{\mu(B(x,r))}\int_{B(x,r)}f(y)d\mu(y),$$
where $B(x,r)=\{y\in X:\;d(x,y)\leq r\}$.

We shall study the weak type $(1,1)$ norm of the maximal operator,
defined to be the least quantity $c_1$ such that for all $f\in
L^+_1(X; S_1)$, all $\lambda>0$, there exists a projection
$e\in\mathcal{P}(L_{\infty}(X)\bar{\otimes} \B(\ell_2))$ satisfying
\begin{align}\label{1 estimates}
ef_re\leq\lambda,\;\forall
r>0\;\;\mathrm{and}\;\;\mathrm{tr}\otimes\int e^{\perp}\leq
\frac{c_1\|f\|_1}{\lambda}.
\end{align}
Here $L_p(X;S_p)$ denotes the noncommutative $L_p$ spaces associated
with von Neumann algebra $\A=L_{\infty}(X)\bar{\otimes} \B(\ell_2)$,
which is the weak closure of the algebra formed by essentially
bounded functions $f:\;X\;\rightarrow\;\B(\ell_2)$. $L^+_p(\A)$ is
the positive part of $L_p(\A)$. $\mathcal{P}(\A)$ denotes the set of
all projections in $\A$.

Analogously to (\ref{1 estimates}), the strong $(p,p)$ norm of the
maximal operator is defined to be the least quantity $c_p$ such that
for all $f\in L^+_p(\A)$, there exists $F\in L^+_p(\A)$ satisfying
\begin{align}\label{p estimates}
f_r\leq F,\;\forall r>0\;\;\mathrm{and}\;\;\|F\|_p\leq c_p\|f\|_p.
\end{align}

In the scalar-valued case, i.e. replacing $\B(\ell_2)$ by complex
numbers $\mathbb{C}$, $c_1$ and $c_p$ are reduced to be the weak
$(1,1)$-boundedness and $L_p$-boundedness of the Hardy-Littlewood
maximal function
$$M(f)(x)=\sup_{r>0}\frac{1}{|B(x,r)|}\int_{B(x,r)}|f(y)|d\mu(y).$$
This maximal function seems no available for matrix-valued function
since we can not compare any two matrices or operators, which is one
source of difficulties in the noncommutative analysis. The obstacle
has been successfully overcome by the interaction with operator
space theory. For instance, Junge in \cite{Jun02} formulated
noncommutative Doob's maximal inequality using Pisier's theory of
vector-valued noncommutative $L_p$-space \cite{Pis98}. Later, in
\cite{JuXu06}, Junge and Xu developed a quite involved
noncommutative version of Macinkiewiz interpolation theorem.
Together with Yeadon's weak type $(1,1)$ maximal ergodic inequality,
the interpolation result enable them to establish a noncommutative
analogue of the Dunford-Schwartz maximal ergodic inequality. The
noncommutative Stein's maximal ergodic inequality has also been
obtained in the same paper.

Inspired by the maximal inequalities established in the theory of
noncommutative martingale and in the ergodic theory, Mei in
\cite{Mei07} considered the operator-valued Hardy-Littlewood maximal
inequality in $\R^n$. He made use of the geometric property of
$\R^n$ to reduce Hardy-Littlewood maximal inequality to several
operator-valued martingale inequalities, which can be viewed as
Junge's noncommutative Doob's maximal inequality or Cuculescu's weak
type $(1,1)$ inequality for noncommutative martingales. Mei's
inequality is exploited by Chen, Xu and Yin in \cite{CXY} to prove
maximal inequalities associated to the integrable rapidly decreasing
functions.

The reduction method in Mei's arguments inevitably yields that the
constants grow exponentially in $n$, the dimension of the base space
$\R^n$. However, it is well known that the constants $c_p$ when
$p>1$ can be taken to be independent of $n$ in the scalar-valued
case. The first result on this topic belongs to Stein \cite{Ste82}
(see also the appendix of \cite{StSt83}), which asserts that when
$X$ is the $n$-dimensional Hilbert space and $\mu$ is Lebesgue
measure, $c_p$ ($p>1$) can be taken to be independent of $n$. For
general $n$-dimensional normed spaces, Bourgain \cite{Bou861}
\cite{Bou86} and Carbery \cite{Car86} proved that $c_p\leq
C(p)<\infty$ provided $p>3/2$. It is unknown whether or not there is
some $1<p<3/2$ for which there exist $n$-dimensional normed spaces
$X_n$ such that $c_p$ are unbounded.  Bourgain in \cite{Bou87} showed that
$c_p\leq C(p,q)<\infty$ for all $p>1$ when $X=\ell^n_q$ and $q$ is
an even integer, which was extended by M\"uller to $X=\ell^n_q$ for
all $1\leq q<\infty$. Finally, Bougain \cite{Bou12} proved that $c_p<\infty$
for all  $X=\ell^n_{\infty}$. We refer the readers to the Introduction of
\cite{NaTa09} for an overall review of the related results.

A dimension independent bound on $c_p$ would mean that the
operator-valued Hardy-Littlewood maximal inequality is in essence an
infinite dimensional phenomenon. In the scalar-valued case, Stein's
dimension independent bound on $c_p$ ($p>1$) has been exploited by
Ti$\mathrm{\check{s}}$er in \cite{Tis88} to study differentiation of
integrals with respect to certain Gaussian measures on Hilbert
space. Therefore, it is reasonable for us to expect a similar
application of the dimension independent bounds in the
operator-valued case. Moreover, even though many operator-valued
results are motivated by quantum analysis or probability (see e.g.
\cite{Mei07}, \cite{HoYi}, \cite{Par09}, \cite{MePa09},
\cite{HSMP}), some of them are inversely used to study analysis on
some noncommutative structures. For instance, in \cite{CXY}, the
authors studied harmonic analysis on quantum torus through
operator-valued harmonic analysis by transference technique; Junge
Mei and Parcet \cite{JMP1} reduced the analysis on the Fourier
multiplier on discret group von Neumann algebras to operator-valued
results through Junge's cross product techniques. Hence there would
exist some applications of the dimension independent operator-valued
results to the analysis on some noncommutative structures. Last but
not least, the dimension free results are particularly of interest
in the noncommutative analysis, since our research object is of
infinite dimension such as a von Neumann algebra.

In this work, as the first attempt, we restrict us to study the
behavior of operator-valued maximal inequality on $n$-dimensional
Hilbert space equipped with Lebesgue measure. An underlying
principle is that even though there are many difficulties in
transferring classical results to the  operator-valued setting (or
even noncommutative setting), the metric or geometric properties of
the defined spaces may interplay well with the noncommutativity of
the range spaces, as happened in \cite{Mei07}, \cite{HoYi},
\cite{Par09}. The first result in the paper is on the estimates of
$c_1$.
\begin{theorem}
Let $f\in L^+_1(\A)$. Then for any $\lambda>0$, there exists a
universal constant $C$ and a projection $e\in
\mathcal{P}(\mathcal{A})$ such that
$$ef_re\leq \lambda,\;\forall r>0\;\;\mathrm{and}\;\;\mathrm{tr}\otimes\int e^{\perp}\leq \frac{Cn\|f\|_1}{\lambda}.$$
\end{theorem}
This result is a generalization of the one by Stein and
Str\"omberg. The main ingredient of the proof is Yeadon's noncommutative maximal ergodic theorem
\cite{Yea77}  (see also below Lemma \ref{yeadon's theorem}). One
will find a detailed proof in section 3.

The main result of this paper is the following dimension independent
estimates of $c_p$ for $p>1$.
\begin{theorem}\label{main theorem}
Let $1<p\leq\infty$ and $f\in L^+_p(\A)$. Then there exist a
constant $C_p$ which depends only on $p$ but not on $n$, and a
function $F\in L_p(\A)$ such that
$$f_r\leq F,\;\forall r>0\;\;\mathrm{and}\;\;\|F\|_p\leq C_p\|f\|_p.$$
\end{theorem}

This is an matrix-valued analogue of Stein and Str\"omberg's result.
We should point out that the previous two theorems are also true by
replacing $B(\ell_2)$ with any von Neumann algebra equipped with a
trace. But for simplicity, we only prove them in the matrix-valued
case. In section 4, we prove Theorem \ref{main theorem}. The main
idea is due to Stein and Str\"omberg, but we should make use of
the techniques and tools developped recently in the  noncommutative analysis. A key
ingredient in Stein's argument is the spherical maximal inequality.
In section 5, we prove an operator-valued version. In a forthcoming
paper\cite{Hon}, we prove a noncommutative version of Nevo and
Thangavelu's ergodic theorems for radial averages on the Heisenberg
Group \cite{NeTh97}, and this spherical maximal inequality can be
viewed as a special case of this kind of maximal ergodic
inequalities.

Since this paper depends heavily on noncommutative maximal ergodic
inequalities and noncommutative Marcinkiewicz interpolation theorem,
and some readers may not be familiar with the main results or its
related notations, we shall recall part of them in section 2.
Throughout this paper, $C$ denotes a universal constant, may varying
from line to line.

\section{Preliminaries}

We first recall the definition of the noncommutative maximal norm
introduced by Pisier \cite{Pis98} and Junge \cite{Jun02}. Let $\M$
be a von Neumann algebra equipped with a normal semifinite faithful
trace $\tau.$ Let $1\leq p\leq\infty$. We define
$L_p(\M;\ell_\infty)$ to be the space of all sequences
$x=(x_n)_{n\ge1}$ in $L_p(\M)$ which admit a factorization of the
following form: there exist $a, b\in L_{2p}(\M)$ and a bounded
sequence $y=(y_n)$ in $L_\infty(\M)$ such that
 $$x_n=ay_nb, \quad \forall\; n\geq1.$$
The norm of  $x$ in $L_p(\M;\ell_\infty)$ is given by
 $$\|x\|_{L_p(\M;\ell_\infty)}=\inf\big\{\|a\|_{2p}\,
 \sup_{n\geq1}\|y_n\|_\infty\,\|b\|_{2p}\big\} ,$$
where the infimum runs over all factorizations of $x$ as above.

We will follow the convention adopted in  \cite{JuXu06}  that
$\|x\|_{L_p(\M;\ell_\infty)}$ is denoted by
 $\big\|\sup_n^+x_n\big\|_p\ .$ We should warn the reader that
$\big\|\sup^+_nx_n\big\|_p$ is just a notation since $\sup_nx_n$
does not make any sense in the noncommutative setting. We find,
however, that $\big\|\sup^+_nx_n\big\|_p$ is more intuitive than
$\|x\|_{L_p(\M;\ell_\infty)}$. The introduction of this notation is
partly justified by the following remark.
\medskip

\begin{remark}
Let $x=(x_n)$ be a sequence of selfadjoint operators in $L_p(\M)$.
Then $x\in L_p(\M;\ell_\infty)$ iff there exists a positive element
$a\in L_p(\M)$ such that $-a\leq x_n \le a$ for all $n\ge1$. In this
case we have
 $$\big \|{\sup_{n \ge1}}^{+} x_n \big \|_p=\inf\big\{\|a\|_p\;:\; a\in L_p(\M),\; -a\leq x_n\le a,\;\forall\; n\geq1\big\}.$$
\end{remark}

More generally, if $\Lambda$ is any index set, we define  $L_p (\M;
\ell_{\infty}(\Lambda))$ as the space of all $x =
(x_{\lambda})_{\lambda \in \Lambda}$ in $L_p (\M)$ that can be
factorized as
 $$x_{\lambda}=ay_{\lambda} b\quad\mbox{with}\quad a, b\in L_{2p}(\M),\; y_{\lambda}\in L_\infty(\M),\; \sup_{\lambda}\|y_{\lambda}\|_\infty<\infty.$$
The norm of $L_p (\M; \ell_{\infty}(\Lambda))$ is defined by
 $$\big \| {\sup_{{\lambda}\in\Lambda}}^{+} x_{\lambda}\big \|_p=\inf_{x_{\lambda}=ay_{\lambda} b}\big\{\|a\|_{2p}\,
 \sup_{{\lambda}\in\Lambda}\|y_{\lambda}\|_\infty\,\|b\|_{2p}\big\} .$$
It is shown in \cite{JuXu06} that $x\in L_p (\M;
\ell_{\infty}(\Lambda))$ iff
 $$\sup\big\{\big \| {\sup_{{\lambda}\in J}}^{+} x_{\lambda}\big \|_p\;:\; J\subset\Lambda,\; J\textrm{ finite}\big\}<\infty.$$
In this case, $\big \| {\sup_{{\lambda}\in\Lambda}}^{+}
x_{\lambda}\big \|_p$ is equal to the above supremum.

A closely related operator space is $L_p(\mathcal{M};\ell^c_{\infty})$ for $p\geq2$ which is the set of all sequences $(x_n)_n\subset L_p(\mathcal{M})$ such that
$$\|{\sup_{n\geq1}}^+|x_n|^2\|^{1/2}_{p/2}<\infty.$$
While $L_p(\mathcal{M};\ell^r_{\infty})$ for $p\geq2$ is the Banach space of all sequences $(x_n)_n\subset L_p(\mathcal{M})$ such that $(x^*_n)_n\in L_p(\mathcal{M};\ell^c_{\infty})$. All these spaces fall into the scope of amalgamated $L_p$ spaces intensively studied in \cite{JuPa10}. What we need about these spaces is the following interpolation results.
\begin{lemma}\label{l8c+l8r}
Let $2\leq p\leq\infty$. Then we have
$$(L_{p}(\mathcal{M};\ell^c_{\infty}),L_{p}(\mathcal{M};\ell^r_{\infty}))_{1/2}=L_{p}(\mathcal{M};\ell_{\infty})$$
with equivalent norms.
\end{lemma}
We refer the reader to \cite{Jun02}, \cite{Mus03} and \cite{JuPa10} for more properties on these spaces. 

Yeadon's weak type $(1,1)$ maximal ergodic inequality for semigroup
is stated as follows:
\begin{lemma}\label{yeadon's theorem}
Let $(T_t)_{t\geq0}$ be a semigroup of linear maps on $\M$. Each
$T_t$ for $t\geq0$ satisfies the following properties:
\begin{enumerate}[\rm (i)]
\item $T_t$ is a contraction on $\M$:
$\|Tx\|_{\infty}\leq\|x\|_{\infty}$ for all $x\in\M$;
\item $T_t$ is positive: $Tx\geq0$ if $x\geq0$;
\item $\tau\circ T\leq\tau$: $\tau(T(x))\leq\tau(x)$ for all $x\in
L_1(\M)\cap\M^+$.
\end{enumerate}
Let $x\in L_1^+(\M)$, then for any $\lambda>0$, there exists a
projection $e\in\mathcal{\M}$ such that
$$eM_t(x)e\leq\lambda,\;\forall t>0,\;\;\mathrm{and}\;\;\tau(e^{\perp})\leq\frac{\|x\|_1}{\lambda},$$
where $M_t$ is defined as
$$M_t=\frac{1}{t}\int^t_0T^sds,\;\;\forall t>0.$$
\end{lemma}

In order to extend this result to $p>1$, Junge and Xu \cite{JuXu06}
proved the following much involved noncommutative Marcinkiewicz
theorem for $L_p(\M;\ell_{\infty})$.
\begin{lemma}\label{marcinkiewicz theorem}
Let $1\leq p_0<p_1\leq\infty$. Let $S=(S_n)_{n\geq0}$ be a sequence
of maps from $L^+_{p_0}(\M)+L^+_{p_1}(\M)$ into $L^+_0(\M)$. Assume
that $S$ is subadditive in the sense that $S_n(x+y)\leq
S_n(x)+S_n(y)$ for all $n\in\mathbb{N}$. If $S$ is of weak type
$(p_0,p_0)$ with constant $C_0$ and of type $C_1$, then for any
$p_0<p<p_1$, $S$ is of type $(p,p)$ with constant $C_p$ satisfying
$$C_p\leq CC_0^{1-\theta}C_1^{\theta}(\frac{1}{p_0}-\frac{1}{p})^{-2},$$
where $\theta$ is determined by $1/p=(1-\theta)/{p_0}+\theta/{p_1}$
and $C$ is a universal constant.
\end{lemma}

With this interpolation result, they proved that there exists a
constant $C_p$ such that \begin{align}\label{p estimate for
Mt}\|{\sup_{t>0}}^+ M_t(x)\|_p\leq C_p\|x\|_p,\;\;\forall x\in
L_p(\M).\end{align} Moreover, if additionally each $T_t$ satisfies
\begin{enumerate}[\rm (iv)]
\item $T_t$ is symmetric relative to $\tau$:
$\tau(T(y)^*x)=\tau(y^*T(x))$ for all $x,\; y$ in the intersection
$L_2(\M)\cap\M$,
\end{enumerate}
then
\begin{align}\label{p estimate for Tt}
\|{\sup_{t>0}}^+ T_t(x)\|_p\leq C_p\|x\|_p,\;\;\forall x\in
L_p(\M),
\end{align}
with $C_p$ a constant only depending on $p$.

\section{Estimates for $c_1$}

We follow Stein and Str\"omberg's original argument to prove Theorem
\ref{1 estimates}. As we shall see that it is just an application of
Yeadon's weak type $(1,1)$ noncommutative maximal ergodic
inequality.

\begin{proof}
Let $f\in L_1(\mathbb{R}^n;L_1(\M))$. Without loss of generality, we
assume $f$ is positive. We then define
$$f_r(x)=\frac{1}{|B(0,r)|}\int_{B(0,r)}f(x-y)dy.$$

Recall that the heat-diffusion semigroup on $\mathbb{R}^n$ is given
by $T^tg=g\ast h_t$, $\forall g\in \mathscr{S}(\mathbb{R}^n)$ with
$$h_t(x)=\frac{1}{(4\pi t)^{n/2}}e^{-|x|^2/{4t}}.$$
We consider the heat-diffusion semigroup on
$L_{\infty}(\mathbb{R}^n)\bar{\otimes}\M$ given by $S^t=T^t\otimes
id_{\M}$. It is easy to check that $(S^t)_{t\geq0}$ satisfies
(i)-(iii). So by Lemma \ref{yeadon's theorem}, for any $\eta>0$,
there exists a projection $e\in \mathcal{P}(\A)$ such that
$$eM_t(f)e\leq\eta,\;\forall t>0,\;\;\mathrm{and}\;\;\mathrm{tr}\otimes\int e^{\perp}\leq \frac{\|f\|_1}{\eta},$$
where $$M_t(f)=\frac{1}{t}\int^t_0
S^s(f)(x)ds=\int_{\R^n}\frac{1}{t}\int^t_0h_s(y)dsf(x-y)dy.$$ As
proved in Page 265 of \cite{StSt83}, for any $r>0$, there exists
some $t_r$ such that
\begin{align}\label{kernel average estimates}
\frac{1}{|B(0,r)|}\chi_{B(0,r)(y)}\leq
Cn\frac{1}{t_r}\int^{t_r}_0h_s(y)ds. \end{align} Hence, obviously we
have
$$ef_re\leq eCnM_{t_r}(f)e\leq Cn\eta.$$
Now for any $\lambda>0$, take $\eta=\lambda/(Cn)$, we obtain
$$\mathrm{tr}\otimes\int e^{\perp}\leq \frac{Cn\|f\|_1}{\lambda},$$
which finishes the proof.
\end{proof}

Instead of using Yeadon's inequality, but use Junge and Xu's
inequality (\ref{p estimates}), in the same spirit, we can deduce
that for $1<p\leq\infty$, there exists an absolute constant $C_p>0$
such that
\begin{align}\label{part of main theorem}
\|{\sup_{r>0}}^+f_r\|_p\leq C_pn\|f\|_p,\;\;\forall f\in L_p(\A).
\end{align}
And the constant $c_p$ can be improved to be $O(\sqrt{n})$ by the
noncommutative Stein's maximal ergodic inequality (\ref{p estimate
for Tt}) and the following fundamental estimates \cite{StSt83}: for
any $r>0$, there is $t_r>0$ such that
\begin{align}
\frac{1}{|B(0,r)|}\chi_{B(0,r)(y)}\leq Cn^{1/2}h_{t_r}(y).
\end{align}

\section{Estimates for $c_p$ ($p>1$)}

We adapt Stein's argument \cite{Ste82} (see also the appendix of
\cite{StSt83}) to the operator-valued setting. The key step of the
argument is the following operator-valued spherical maximal
inequality. Let $f\in \mathscr{S}(\R^k;S_{M})$ ($S_{M}$ is the set
of finite dimension self-adjoint matrix), for any $r>0$, we define
$$f^k_r(x)=\frac{1}{\omega_{k-1}}\int_{S^{k-1}}f(x-ry')d\sigma(y'),$$
where $d\sigma$ is the usual measure on $S^{k-1}$, and
$\omega_{k-1}$ is its total mass.

\begin{proposition}\label{spherical maximal inequality}
Let $k\geq3$ and $p>k/(k-1)$, then there exists a constant $A_{k,p}$
such that
$$\|{\sup_{r>0}}^+{f^{k}_r}\|_p\leq A_{k,p}\|f\|_p, \;\; \forall f\in L_p(\R^k;S_p).$$
\end{proposition}

We postpone its proof to the next section. The spherical maximal
inequality yields the following weighted maximal inequality. Let
$f\in\mathscr{S}(\R^k;S_{M})$, for any $m\geq0$ and $r\geq0$, we
define
$$f^{k,m}_r(x)=(\int_{|y|\leq r}|y|^mdy)^{-1}\int_{|y|\leq r}f(x-y)|y|^mdy.$$
\begin{proposition}\label{weighted maximal inequality}
Let $k\geq3$ and $p>k/(k-1)$, then
$$\|{\sup_{r>0}}^+{f^{k,m}_r}\|_p\leq A_{k,p}\|f\|_p, \;\; \forall f\in L_p(\R^k;S_p)$$
with the constant $A_{k,p}$ independent of $m$.
\end{proposition}

\begin{proof}
Without loss of generality, we assume $f\in
\mathscr{S}(\R^k;S^+_{M})$. Using polar coordinates, we can write
\begin{align}\label{polar decomposition}
\int_{|y|\leq
r}f(x-y)|y|^mdy=\int^r_0\int_{S^{k-1}}f(x-sy')s^{m+k-1}d\sigma(y')ds.
\end{align}
By Proposition \ref{spherical maximal inequality}, there exists
$F\in L_p^+(\R^k;S_p)$ such that
$$f^k_s(x)\leq F(x),\;\forall s>0\;\;\mathrm{and}\;\;\|F\|_p\leq A_{k,p}\|f\|_p.$$
Hence
$$\mathrm{RHS}\;\mathrm{of}\;(\ref{polar decomposition})\leq F(x)\omega_{k-1}\int^r_0s^{m+k-1}ds=F(x)\omega_{k-1}\frac{r^{m+k}}{m+k}.$$
So we have
$$f^{k,m}_r\leq F,\;\forall r>0,\;\;\mathrm{and}\;\;\|F\|_p\leq A_{k,p}\|f\|_p$$
which is the desired result.
\end{proof}

We now consider $\R^n$ with $n\geq3$, and write it as
$\R^n=\R^k\times\R^{n-k}$ with its points $x$ written by
$(x_1,x_2)$. Let $\rho$ denote an arbitrary element of $O(n)$, a
rotation of $\R^n$ about the origin. Let
$f\in\mathscr{S}(\R^n;S_{\M})$, for each $\rho\in O(n)$, $r>0$, we
define
$$f^{k,n-k,\rho}_r(x)=(\int_{|y_1|\leq r}|y_1|^{n-k}dy_1)^{-1}\int_{|y_1|\leq r}f(x-\rho(y_1,0))|y_1|^{n-k}dy_1.$$
\begin{proposition}\label{directional maximal inequality}
Let $k\geq3$ and $p> k/(k-1)$, we have
$$\|{\sup_{r>0}}^+f^{k,n-k,\rho}_r\|_p\leq A_{k,p}\|f\|_p,\;\;\forall f\in L_p(\A)$$
with the constant $A_{k,p}$ independent of $n$.
\end{proposition}
\begin{proof}
Take $f\in \mathscr{S}(\R^n;S_{M})$. Again, we assume $f$ is
positive. By rotation invariance, it suffices to prove this when
$\rho$ is the identity rotation. In this case, we decompose
$\R^n=\R^k\times\R^{n-k}$, with $x=(x_1,x_2)$. Fix $x_2\in\R^{n-k}$.
By Proposition \ref{weighted maximal inequality}, there exist
$F_{x_2}\in L^+_p(\R^k;S_p)$ such that
$$f^{k,n-k,1}_r(x_1,x_2)\leq F_{x_2}(x_1)\;\forall r>0,\;\;\mathrm{and}\;\;\|F_{x_2}\|_p\leq A_{k,p}\|f_{x_2}\|_p.$$
Define $F(x_1,x_2)=F_{x_2}(x_1)$ on $\R^n$, then we complete the
proof since $f^{k,n-k,1}_r\leq F$ for all $r>0$ and
\begin{align*}
\|F\|^p_p&=\int_{\R^{n-k}}\|F(\cdot,x_2)\|^p_pdx_2\\
&\leq
A^p_{k,p}\int_{\R^{n-k}}\|f(\cdot,x_2)\|^p_pdx_2=A^p_{k,p}\|f\|^p_p\\
\end{align*}
\end{proof}

Let $d\rho$ denote the Haar measure on the group $O(n)$, normalized
so that its total measure is $1$. Now we are at a position to prove
Theorem \ref{main theorem}.
\begin{proof}
The result for $p=\infty$ is trivial. So we only consider the case
$1<p<\infty$. When $n\leq \max(p/(p-1),2)$, we can use the estimates
(\ref{part of main theorem}). Now, we assume $n>\max(p/(p-1),2)$. We
write $n=k+(n-k)$, where $k$ is the smallest integer greater than
$\max(p/(p-1),2)$. We can assume $f$ is of the form $g\otimes m$
where $g\in\mathscr{S}^+(\R^n)$ and $m\in S^+_{M}$, since the set of
linear combinations of such elements are dense in $L_p(\A)$. For
such $f$, we have the following formula
\begin{align}\label{key identity}
\frac{\int_{|y|\leq r}f(y)dy}{\int_{|y|\leq
r}dy}=\frac{\int_{O(n)}\int_{|y_1|\leq
r}f(\rho(y_1,0))|y_1|^{n-k}dy_1d\rho}{\int_{|y_1|\leq
r}|y_1|^{n-k}dy_1}.
\end{align}
Here $y=(y_1,y_2)\in\R^n=\R^k\times\R^{n-k}$. To verify (\ref{key
identity}) it suffices to do so for $g$ of the form
$g=g_0(|y|)g_1(y')$, where $y'\in S^{n-1}$, and $y=|y|y'$, since
linear combination of such functions are dense. Then for such $g$,
$$\mathrm{LHS}\;\mathrm{of}\;(\ref{key identity})=\int^r_0g_0(t)t^{n-1}dt\cdot\int_{S^{n-1}}g_1(y')d\sigma(y')nr^{-n}\omega^{-1}_{n-1}\otimes m.$$
On the other hand, notice that
$g(\rho(y_1,0))=g_0(|y_1|)g_1(\rho(y_1,0))$, so the right hand side
of (\ref{key identity}) equals
$$\int^r_0g_0(t)t^{n-1}dt\cdot\int_{O(n)}\int_{S^{k-1}}g_1(\rho(y'_1,0)d\sigma(y_1')d\rho nr^{-n}\omega^{-1}_{k-1}\otimes m.$$
Therefore matters are reduced to check that
$$\frac{1}{\omega_{n-1}}\int_{S^{n-1}}g_1(y')d\sigma(y')=\frac{1}{\omega_{k-1}}\int_{O(n)}\int_{S^{k-1}}g_1(\rho(y'_1,0))d\sigma(y_1')d\rho$$
which is trivial because
$$\int_{O(n)}g_1(\rho(y'_1,0))d\rho=\frac{1}{\omega_{n-1}}\int_{S^{n-1}}g_1(y')d\sigma(y'_1).$$
In (\ref{key identity}), replace $f(y)$ with $f(x-y)$, we get
\begin{align*}
\frac{\int_{|y|\leq r}f(x-y)dy}{\int_{|y|\leq
r}dy}=\frac{\int_{O(n)}\int_{|y_1|\leq
r}f(x-\rho(y_1,0))|y_1|^{n-k}dy_1d\rho}{\int_{|y_1|\leq
r}|y_1|^{n-k}dy_1}.
\end{align*}
By Proposition \ref{directional maximal inequality}, for each
$\rho\in O(n)$, there exists $F^{\rho}\in L^+_p(\A)$ such that
$$f^{k,n-k,\rho}_r\leq F^{\rho},\;\forall r>0,\;\;\mathrm{and}\;\;\|F^{\rho}\|_p\leq A_{k,p}\|f\|_p.$$
Hence one can very easily deduce that
$F(x)=\int_{O(n)}F^{\rho}(x)d\rho$ is in $L_p(\A)$ and satisfy
$$f^{k,n-k}_r\leq F,\;\forall r>0,\;\;\mathrm{and}\;\;\|F\|_p\leq A_{k,p}\|f\|_p.$$
\end{proof}

\section{The proof of Proposition \ref{spherical maximal inequality}}
In order to simplify the notation, we denote  $f^n_t$ by
$f_t/{\omega_{n-1}}$. Hence
$$f_t(x)=\int_{S^{n-1}}f(x-t\theta)d\sigma(\theta).$$
We set
$m(\xi)=\widehat{d\sigma}(\xi)=2\pi|\xi|^{(2-n)/2}J_{(n-2)/2}(2\pi|\xi|)$
(see e.g. Appendix B.4 in \cite{Gra08}). Obviously $m(\xi)$ is an
infinitely differential function. We decompose the multiplier
$m(\xi)$ into radial pieces as follows: We fix a radial Schwartz
function $\varphi_0$ in $\mathbb{R}^n$ such that $\varphi_0(\xi)=1$
when $|\xi|\leq1$ and $\varphi_0(\xi)=0$ when $|\xi|\geq2$. For
$j\geq1$, we let
$$\varphi_j(\xi)=\varphi_0(2^{-j}\xi)-\varphi_0(2^{1-j}\xi)$$
and we observe that $\varphi_j(\xi)$ is localized near $|\xi|=2^j$.
Then we have
$$\sum_{j\geq0}\varphi_j=1.$$
Set $m_j=\varphi_jm$ for all $j\geq0$. The $m_j$'s are finite
supported Schwartz functions that satisfy
$$m=\sum_{j\geq0}m_j.$$
Hence,
\begin{align*}
f_t(x)=(\hat{f}(\cdot)m(t\cdot))^{\vee}
=\sum_{j\geq0}(\hat{f}(\cdot)m_j(t\cdot))^{\vee}=\sum_{j\geq0}f_{t,j}
\end{align*}
For these $f_{t,j}$, there are the following estimates.

\begin{proposition}\label{easy estimates}
Let $1<p\leq\infty$, there exists a constant $C=C(n,p)$ such that
$$\|{\sup_t}^+ f_{t,0}\|_p\leq C\|f\|_p.$$
More precisely, for $f\in L^+_p(\A)$, there exists $F_0\in L_p(\A)$
such that \begin{align}\label{inequality 1} f_{t,0}\leq
F_0,\;\forall t>0\;\;\mathrm{and}\;\;\|F_0\|_p\leq C\|f\|_p.
\end{align}
\end{proposition}

\begin{proposition}\label{main estimates}
Let $1<p\leq 2$. There exists a universal constant $C=C(n,p)$ such
that for any $j\geq1$, we have
$$\|{\sup_t}^{+} f_{t,j}\|_p\leq C2^{(n/p-(n-1))j}\|f\|_p,\;\;\;\forall f\in L_p(\A).$$
More precisely, for $f\in L^+_p(\A)$, there exists $F_j\in L_p(\A)$
such that
\begin{align}\label{inequality 2}
f_{t,j}\leq F_j,\;\forall t>0\;\;\mathrm{and}\;\;\|F_j\|_p\leq
C2^{(n/p-(n-1))j}\|f\|_p.
\end{align}
\end{proposition}

With the two previous estimates, we can finish the proof of
Proposition \ref{spherical maximal inequality}.

\begin{proof}
Let $f\in L^+_p(\A)$. When $2\geq p>n/(n-1)$, by Proposition
\ref{easy estimates} and \ref{main estimates}, we find $F_j$'s
satisfying inequality (\ref{inequality 1}) or (\ref{inequality 2}).
We set
$$F=\sum_{j\geq0}F_j.$$
Then
\begin{align*}
f_t=\sum_{j\geq0}f_{t,j}\leq \sum_{j\geq0}F_j=F,\;\forall t>0
\end{align*}
and
\begin{align*}
\|F\|\leq\sum_{j\geq0}\|F_j\|_p\leq
C\sum_{j\geq0}2^{(n/p-(n-1))j}\|f\|_p=C\|f\|_p.
\end{align*}
When $p\geq 2$, we invoke the noncommutative interpolation theorem ,
Lemma \ref{marcinkiewicz theorem} to obtain the estimates.
\end{proof}

The rest of this section is devoted to the proof of the two
propositions. Proposition \ref{easy estimates} is a trivial
application of the following Theorem 4.3 of \cite{CXY}.

\begin{lemma}\label{result of cxy}
  Let $\psi$ be an integrable function on $\R^n$ such that $|\psi|$ is radial and radially decreasing.
 Let $\psi_t(x)=\frac1{t^n}\, \psi(\frac xt)$ for $x\in\R^n$ and $t>0$.
 \begin{enumerate}[{\rm i)}]
 \item  Let $f\in L_1(\R^n; S_1)$. Then for  any $\alpha>0$ there exists a projection $e\in \mathcal{P}(\A)\overline{\otimes}\M$ such that
 $$\sup_{t>0}\big\|e(\psi_t\ast f)e\big\|_\infty\leq\alpha\quad\text{and}\quad \mathrm{tr}\otimes\int e^{\perp}\leq C_n\|\psi\|_1\frac{\|f\|_1}\alpha.$$

\item  Let $1<p\leq\infty$. Then
 $$\big\|{\sup_{t>0}}^+\psi_t*f\big\|_p\le C_n\|\psi\|_1\,\frac{p^2}{ (p-1)^2}\|f\|_p,\quad\forall\; f\in L_p(\R^n; S_p)).$$
\end{enumerate}
\end{lemma}

On the proof of Proposition \ref{main estimates},  it suffices to establish the two end-point estimates $p=2$ and $p=1$, since a noncommutative version of Marcinkiewicz interpolation theorem is available (see Lemma
\ref{marcinkiewicz theorem}). 
\begin{lemma}
There exists a constant $C=C(n)<\infty$ such that for any $j\geq1$
we have
\begin{align*}
\|{\sup_{t>0}}^+f_{t,j}\|_2\leq
C2^{(1/2-(n-1)/2)j}\|f\|_2,\;\;\forall f\in L_2(\A).
\end{align*}
\end{lemma}

\begin{proof}
 We define a function
$$\tilde{m}_j(\xi)=\xi\cdot\triangledown m_j(\xi).$$ Let
$$\tilde{f}_{t,j}(x)=(\hat{f}(\cdot)\tilde{m}_j(t\cdot))^{\vee}(x).$$
And we consider the following two $g$-functions:
$$G_j(f)(x)=\big(\int^{\infty}_0|f_{t,j}(x)|^2\frac{dt}{t}\big)^{\frac12},$$
and
$$\tilde{G}_j(f)(x)=\big(\int^{\infty}_0|\tilde{f}_{t,j}(x)|^2\frac{dt}{t}\big)^{\frac12}.$$
For $f\in\mathscr{S}(\mathbb{R}^n,S^+_{\M})$, the identity
$$s\frac{df_{s,j}}{ds}=\tilde{f}_{s,j}$$
hold for all $j$ and $s$.
By the fundamental theorem of calculus, we deduce that
\begin{align*}
{f_{t,j}(x)}^2&=\int^t_{\varepsilon}\frac{d}{ds}(f_{s,j}(x))^2ds+f_{\varepsilon,j}(x)^2\\
&=\int^t_{\varepsilon}s\frac{df^*_{s,j}(x)}{ds}f_{s,j}(x)f^*_{s,j}(x)s\frac{df_{s,j}(x)}{ds}\frac{ds}{s}+f_{\varepsilon,j}(x)^2\\
&=\int^t_{\varepsilon}\tilde{f}^*_{s,j}(x)f_{s,j}(x)+f^*_{s,j}(x)\tilde{f}_{s,j}(x)\frac{ds}{s}+f_{\varepsilon,j}(x)^2\\
&\leq\int^{\infty}_{0}|\tilde{f}^*_{s,j}(x)f_{s,j}(x)\frac{ds}{s}+\int^{\infty}_0f^*_{s,j}(x)\tilde{f}_{s,j}(x)|\frac{ds}{s}+f_{\varepsilon,j}(x)^2.\\
\end{align*}
Hence by triangle inequality and H\"older inequality, we have
\begin{align*}
\|{\sup_{t}}^+|f_{t,j|^2}\|^{1/2}_1&\leq\|\int^{\infty}_0|\tilde{f}^*_{s,j}(x)f_{s,j}(x)+f^*_{s,j}(x)\tilde{f}_{s,j}(x)|\frac{ds}{s}\|^{1/2}_1\\
&\;\;\;\;\;+\|f_{\varepsilon,j}(x)^2\|^{1/2}_1\\
&\leq2\|\int^{\infty}_0\tilde{f}^*_{s,j}(x)f_{s,j}(x)\frac{ds}{s}\|^{1/2}_1\\
&\;\;\;\;\;+2\|\int^{\infty}_0{f}^*_{s,j}(x)\tilde{f}_{s,j}(x)\frac{ds}{s}\|^{1/2}_1+\|f_{\varepsilon,j}(x)^2\|^{1/2}_1\\
&\leq4\|G_j(f)\|^{\frac 12}_2\|\tilde{G}_j(f)\|^{\frac 12}_2+\|f_{\varepsilon,j}(x)^2\|^{1/2}_1.\\
&\leq8\|G_j(f)\|^{\frac 12}_2\|\tilde{G}_j(f)\|^{\frac 12}_2.\\
\end{align*}
The last inequality is due to the fact that  $\|f_{\varepsilon,j}(x)^2\|^{1/2}_1$ tends to 0 as $\varepsilon$ tends to $\infty$ by Lebesgue dominated theorem.
On the other hand, by the estimates (see e.g. \cite{Gra08})
$$|\hat{d\sigma}(\xi)|+|\triangledown\hat{d\sigma}(\xi)|\leq
C_n(1+|\xi|)^{(1-n)/2},$$ we have
$$\|m(\xi)\|_{\infty}\leq C2^{-j\frac{n-1}{2}}\;\;\mathrm{and}\;\;\|m(\xi)\|_{\infty}\leq C2^{j(1-\frac{n-1}{2})}.$$
Using these elementary estimates and the facts that the functions
$m_j$ and $\tilde{m}_j$ are supported in the annuli around
$|\xi|=2^{j}$, we obtain that these two $g$-functions are
$L_2$-bounded with norms at most a constant multiple of the
quantities $2^{-j\frac{n-1}{2}}$ and $2^{j(1-\frac{n-1}{2})}$
respectively. Hence
$$\|{\sup_{t}}^+f_{t,j}\|_2\leq C2^{j(\frac{1}{2}-\frac{n-1}{2})}\|f\|_2.$$
\end{proof}

\begin{lemma}
There exists a constant $C=C(n)<\infty$ such that for all $j>1$, we
have
$$\|{\sup_{t}}^+f_{t,j}\|_{1,\infty}\leq C2^{j}\|f\|_1,\;\;\forall f\in L_1(\A).$$
More precisely, for all $\lambda>0$, there is a projection
$e\in\mathcal{P}(\A)$ such that
$$\sup_t\|ef_{t,j}e\|_{\infty}\leq\lambda,\;\;\mathrm{and}\;\;\tau\int(e^{\perp})\leq C2^{j}\|f\|_1.$$
\end{lemma}

\begin{proof}
Let $K_j=(\varphi_j)^{\vee}\ast d\sigma=\Phi_{2^{-j}}\ast d\sigma$,
where $\Phi$ is a Schwartz function. Setting
$$K_{j,t}(x)=t^{-n}K_j(t^{-1}x).$$
We have
$$f_{t,j}=K_{j,t}\ast f.$$
On page 399 of \cite{Gra08}, it is shown that for any $M>n$, there
exists $C_{M}<\infty$ such that
$$|K_j(x)|\leq C_M2^j(1+|x|)^{-M}.$$
Then we complete the proof by Lemma \ref{result of cxy}.
\end{proof}


\begin{thebibliography}{0}

\bibitem{Bou861} J. Bourgain, On high-dimensional maximal functions
associated to convex bodies, Amer. J. Math., 108(6):1467-1476, 1986.

\bibitem{Bou86} J. Bourgain, On the $L^p$-bounds for maximal
functions associated to convex bodies in $\R^n$, Israel J. Math.,
54(3):257-265, 1986.

\bibitem{Bou87} J. Bourgain, On dimension free maximal inequalities
for convex symmetric bodies in $\R^n$, In Geometrical aspects of
functional analysis (1985/86), volume 1267 of Lecture Notes in Math,
pages 168-176, Springer, Berlin, 1987.

\bibitem{Bou12} J.Bourgain, On the Hardy-Littlewood maximal function for the cube, arXiv:1212.2661v1, 2012.

\bibitem{Car86} A. Carbery, An almost-orthogonality principle with
applications to maximal functions to maximal functions associated to
convex bodies, Bull. Amer. Math. Soc. (N.S.), 14(2):269-273, 1986.

\bibitem{CXY} Z. Chen, Z. Yin and Q. Xu, Harmonic analysis on quantum tori, arXiv:1206.3358, 2012.


\bibitem{Gra08} L. Grafakos, Classical Fourier Analysis (Graduate Texts in
Mathematics), Springer; 2nd ed. edition, 2008.

\bibitem{Hon} G. Hong, A noncommutative version of Nevo and Thangavelu's ergodic theorems
for radial averages on the Heisenberg Group, in progress.

\bibitem{HoYi} G. Hong and Z. Yin, Wavelet approach to
operator-valued Hardy spaces, to appear in Rev. Mate. Iberoa.

\bibitem{HSMP} G. Hong, L.D. L\'opez-S\'anchez, J. Martell, J. Parcet, Calder\'on-Zygmund operators associated to
matrix-valued kernels, preprint.


\bibitem{Jun02} M. Junge, Doob's inequality for non-commutative
martingales, J. Reine Angew. Math., 549:149-190, 2002.

\bibitem{JMP1} M. Junge, T. Mei and J. Parcet, Smooth Fourier multipliers on group von Neumann algebras, preprint 2011, arXiv Math: 1010.5320.

\bibitem {JuPa10} M. Junge, J. Parcet, Mixed-norm inequalities and operator space Lp embedding theory, Mem. Amer. Math. Soc. \textbf{952}, 2010.

\bibitem{JuXu06} M. Junge and Q. Xu,
Noncommutative maximal ergodic theorems, J. Amer. Math. Soc.
20:385-439, 2006.

\bibitem{Mei07} T. Mei, Operator valued Hardy spaces, Mem. Amer. Math. Soc.
188, 2007.

\bibitem{MePa09} T. Mei and J. Parcet, Pseudo-localization of singular integrals and noncommutative Littlewood-Paley
inequalities, Int. Math. Res. Not., 9:1433-1487, 2009.

\bibitem {Mus03} M. Musat, Interpolation between non-commutative BMO and non-commutative $L_p$-spaces. J. Funct. Anal. \textbf{202}:195-225, 2003.

\bibitem{NaTa09} A. Naor and T. Tao, Random martingales and localization of maximal inequalities,
Journal of Functional Analysis, 29(3):731-779, 2009.

\bibitem{NeTh97} A. Nevo and S. Thangavelu, Pointwise ergodic
theorems for radial averages on the Heisenberg Group, Adv. Math.,
127:307-334, 1997.

\bibitem{Par09} J. Parcet, Pseudo-localization of singular integrals and noncommutative Calder\'on-Zygmund theory, J. Funct. Anal., 256:5509-593, 2009.


\bibitem{Ste82} E.M. Stein, The development of square functions in
the work of A. Zygmund, Bull. Amer. Math. Sco.(N.S.), 7(2):359-367,
1982.

\bibitem{StSt83} E.M. Stein, J.O. Stromberg, Behavior of maximal functions in $\mathbb{R}^n$  for large n, Arkiv
Math. 21:259-269, 1983.

\bibitem{Tis88} J. Ti$\mathrm{\check{s}}$er, Differentiations theorem for
Gaussian measure on Hilbert space, Trans. Amer. Math. Soc.,
308(2):655-666, 1988.

\bibitem{Pis98} G. Pisier, Non-commutative vector-valued $L_p$-space
and completely $p$-summing maps, Ast\'erisque, 247, 1998.

\bibitem{PiXu97} G. Pisier and Q. Xu, Non-commutative martingale
inequalities, Comm. Math. Phys., 189:667-698, 1997.

\bibitem{Yea77} F.J. Yeadon, Ergodic theorems for semifinite von
Neumann algebras I, J. London. Math. Soc. (2), 16(2): 326-332, 1977.

\end{thebibliography}
\end{document}